\documentclass[11pt]{article}
\usepackage{amsfonts}
\usepackage{hyperref}
\usepackage{epsfig}
\usepackage{graphicx}
\usepackage{amsmath}
\usepackage{amssymb}

\newtheorem{theorem}{Theorem}[section]
\newtheorem{definition}[theorem]{Definition}
\newtheorem{proposition}[theorem]{Proposition}

\newtheorem{claim}[theorem]{Claim}

\newtheorem{remark}[theorem]{Remark}

\newtheorem{question}[theorem]{Question}
\newtheorem{observation}[theorem]{Observation}

\newcommand{\R}{\mathcal{R}}
\renewcommand{\Re}{\mathbb{R}}
\renewcommand{\L}{\mathcal{L}}

\newenvironment{proof}[1][Proof]{ \noindent \textbf{#1: }}{$\Box$
\bigskip}

\title{No Krasnoselskii Number for General Sets in $\mathbb{R}^2$}
\author{Chaya Keller\thanks{Department of Computer Science, Ariel University, Israel. \texttt{chayak@ariel.ac.il}. Research partially supported by the Israel Science Foundation (grant no. 1065/20), the Arianne de Rotschild
fellowship, by the Hoffman Leadership program of the Hebrew
University, and by an Advancing Women in Science grant of the
Israel Ministry of Science and
Technology.}
\mbox{ }
and Micha A. Perles\thanks{Einstein Institute of Mathematics, Hebrew University, Jerusalem, Israel.
\texttt{perles@math.huji.ac.il}}
}

\begin{document}

\maketitle

\begin{abstract}
For a family $\mathcal{F}$ of sets in $\mathbb{R}^d$, the
Krasnoselskii number of $\mathcal{F}$ is the smallest $m$ such that
for any $S \in \mathcal{F}$, if every $m$ points of $S$ are visible from a
common point in $S$, then any finite subset of $S$ is visible from
a single point. More than 35 years ago, Peterson asked whether there exists a Krasnoselskii number for general sets in $\mathbb{R}^d$. Excluding results for special cases of sets with strong topological restrictions, the best known result is due to
Breen, who showed that if such a Krasnoselskii number in $\mathbb{R}^2$ exists,
then it is larger than $8$.

In this paper we answer Peterson's question in the negative by showing that there is no
Krasnoselskii number for the family of all sets in $\mathbb{R}^2$.
The proof is non-constructive, and uses transfinite induction and
the well ordering theorem.

In addition, we consider Krasnoselskii numbers with respect to
visibility through polygonal paths of length $ \leq n$, for which an
analogue of Krasnoselskii's theorem was proved by
Magazanik and Perles. We show, by an explicit
construction, that for any $n \geq 2$, there is no
Krasnoselskii number for the family of general sets in
$\mathbb{R}^2$ with respect to visibility through paths of length
$\leq n$. (Here the counterexamples are finite unions of line segments.)
\end{abstract}

\section{Introduction}

\begin{definition}
For a set $S \subset \mathbb{R}^d$ and two points $x,y \in S$, we
say that ``$x$ sees $y$ through $S$'' or ``$y$ is visible from $x$
through $S$'' if $S$ includes the segment $[x,y]$.
A set $S \subset \mathbb{R}^d$ is called ``starshaped'' if there
exists a point $x \in S$ that sees all other points in $S$. A set $S
\subset \mathbb{R}^d$ is called ``finitely starlike'' if any
finite subset of $S$ is visible from a common point.
\end{definition}

One of the best-known applications of Helly's theorem (that is, actually, equivalent to Helly's theorem, see~\cite{Bor77}) is the
following theorem, due to Krasnoselskii~\cite{Krasnoselskii}:
\begin{theorem}[Krasnoselskii]
Let $S$ be an infinite compact set in $\mathbb{R}^d$. If every $d+1$ points of $S$ are visible from a common point, then
$S$ is starshaped.
\end{theorem}
In particular, the case $d=2$ of the theorem asserts that if an art gallery is so shaped that for every three paintings there is
a place where you can stand and see those three, then there is a place where you can stand and see all of the paintings.

Krasnoselskii's theorem does not hold for non-compact sets.
However, all known counter-examples satisfy the
weaker requirement of being finitely starlike. This led
Peterson~\cite{Peterson} to ask
the following:
\begin{question}[Peterson]
Does there exist a Krasnoselskii number $K(d)$ such that for any
infinite set $S \subset \mathbb{R}^d$, if every $K(d)$ points of $S$
are visible from a single point, then $S$ is finitely starlike?
\end{question}

Peterson's question appears in the Handbook of Combinatorics~\cite[Chapter 17, Section~6.1.2, written by Erd\H{o}s and Purdy]{EP95} and in a treatise on open problems in geometry by Croft, Falconer and Guy~\cite[Chapter E1]{Unsolved}.
It was studied in a number of works. A few positive results were obtained, for special cases of sets that satisfy additional topological conditions~\cite{Breen0,Breen1,Breen5}.
Those include Krasnoselskii number 3 for closed sets in the plane, and Krasnoselskii numbers $4$ and $5$ under more complex conditions on the boundary of the set.
On the other hand, for general sets in the plane, Peterson~\cite{Peterson} showed that if $K(2)$ exists, then it
must be greater than $3$. Breen~\cite{Breen1} improved Peterson's result, showing by an explicit example that if $K(2)$
exists, then it must be larger than $8$. For other partial results and related problems, see~\cite{survey} and the references therein.


\medskip In this paper we
answer Peterson's question in the negative
by showing that the Krasnoselskii number $K(2)$ does not exist. Formally, we prove the
following:
\begin{theorem}\label{Thm:Our1}
For any $k \geq 2$, there exists a set $S \subset \mathbb{R}^2$
such that any $2k+3$ points in $S$ are visible from a common point,
but there exist $2k+4$ points in $S$ that are not visible from a
common point.
\end{theorem}

Unlike Breen's construction that leads to an $F_{\sigma}$-set, our proof is non-constructive and uses
transfinite induction and the well-ordering theorem.

\bigskip

In the second part of the paper, we consider a generalization of
the notion of visibility, namely visibility through polygonal
paths.
\begin{definition}
For a set $S \subset \mathbb{R}^d$ and $x,y \in S$, we say that
``$x$ sees $y$ through $S$ by a polygonal path of length $n$'' or
``$y$ is visible from $x$ through $S$ by a polygonal path of
length $n$'' if there exist distinct $x_1,x_2,\ldots,x_{n-1}$ and a polygonal path
$$ \langle x, x_1, \ldots, x_{n-1},y\rangle  =  [x,x_1]\cup[x_1,x_2]\cup \ldots \cup [x_{n-1},y]$$
that lies entirely within $S$.
\end{definition}

This definition of visibility is closely related to the notion of
$L_n$-sets, which are sets that are connected through polygonal
paths of length $\leq n$. This notion was defined by Horn and
Valentine~\cite{Horn}, and studied in numerous papers (see,
e.g.,~\cite{Bruckner,Evelyn1,Valentine}). Of course, the original
notion of visibility considered above corresponds to the case
of visibility through polygonal paths of length $n=1$.

In~\cite[Theorem~1.2]{Evelyn1}, Magazanik and Perles proved a
Krasnoselskii-type theorem with respect to visibility through
polygonal paths. Formally, they showed that for any $n$ and for
any compact and simply connected set $S \subset \mathbb{R}^2$, if any
three points in $S$ are visible from a common point through
polygonal paths of length $\leq n$, then all points of $S$ are visible
from a single point through polygonal paths of length $\leq n$.

This raises the natural analogue of Peterson's question for this
notion of visibility: Is there a number $K'(2,n)$ such
that for any finite set $S \subset \mathbb{R}^2$, if every $K'(2,n)$ points
of $S$ are visible from a common point through polygonal paths of
length $\leq n$, then any finite subset of $S$ is visible from a common
point through polygonal paths of length $\leq n$?

We answer this question in the negative. Namely, we prove,
by constructing a sequence of explicit examples, the following:
\begin{theorem}\label{Thm:Our2}
For any $n,k \geq 2$, there exists a set $S
\subset \mathbb{R}^2$ such that any $k$ points in $S$ are visible
from a common point through polygonal paths of length $\leq n$, but
there exist $k+1$ points in $S$ that are not visible from a common
point through polygonal paths of length $\leq n$.
\end{theorem}

The sets $S$ we construct in the proof of Theorem~\ref{Thm:Our2} are actually finite unions of closed line segments.

A similar result with respect to the related concept of visibility through {\it staircase} paths of length $\leq n$ was obtained by Breen~\cite{Breen4}. However, the construction of~\cite{Breen4} is significantly different from ours.


\medskip This paper is organized as follows: In Section~\ref{sec:Our1} we study Peterson's problem with respect to the classical definition of visibility and prove Theorem~\ref{Thm:Our1}. In Section~\ref{sec:Our2} we study visibility through polygonal paths of length $\leq n$ and prove Theorem~\ref{Thm:Our2}.

\section{Proof of Theorem~\ref{Thm:Our1}}
\label{sec:Our1}


\subsection{Some basic facts from set theory}
\label{sec:sub:set}

In this subsection we recall a few basic definitions and facts from set theory that will be used in our proof. For the sake of brevity, we present the definitions in a somewhat informal way. A formal treatment can be found in standard textbooks on basic set theory, e.g.,~\cite{Levy79}.

\medskip \noindent \emph{Well ordering.}
Let $S$ be a set. A binary relation $<$ on $S$ is called a \emph{well ordering} if:
\begin{enumerate}
	\item For any $x,y \in S$, exactly one of the following holds: either $x<y$, or $y<x$, or $x=y$.
	\item $"<"$ is transitive.
	\item Any non-empty subset $S'$ of $S$ has a first element $x_0$, namely, $\forall \emptyset \neq S' \subset S, \exists x_0 \in S': \forall x \in S', x_0 \leq x$.
\end{enumerate}
A pair $(S,<)$ where $<$ is a well ordering on $S$ is called a \emph{well ordered set}.

\medskip \noindent \emph{Transitive sets and ordinals.} A set $S$ is called \emph{transitive} if any element of $S$ is a subset of $S$, that is, if $(x \in S) \wedge (y \in x) \Rightarrow (y \in S)$. A set $O$ is called an \emph{ordinal} if it is transitive and $(O,\in)$ (i.e., $O$ with the relation $(x<y) \Leftrightarrow (x \in y)$) is a well-ordered set.

\medskip \noindent \emph{Initial section.} Let $(S,<)$ be a well-ordered set. An initial section of $S$ is a subset $S' \subset S$ of the form $S'=\{x \in S:x <y\}$, for some $y \in S$.

\medskip \noindent \emph{Isomorphism.} Two well-ordered sets $(S,<_S)$ and $(T,<_T)$ are \emph{isomorphic} if there exists a bijection $f:S \rightarrow T$ such that $\forall x,y \in S: (x <_S y) \Leftrightarrow (f(x) <_T f(y))$.

\medskip \noindent \emph{Two basic facts on ordinals.}
\begin{enumerate}
	\item For any set $\mathcal{O}$ of ordinals, $(\mathcal{O},\in)$ is a well-ordered set. In particular, any set of ordinals contains a minimal element with respect to the relation $\in$.
	
	\item Any well-ordered set $(S,<)$ is isomorphic to a unique ordinal.
\end{enumerate}

\medskip \noindent \emph{The well ordering theorem.} For any set $S$, there exists a relation $<$ on $S$ such that $(S,<)$ is a well-ordered set. (This theorem may be considered an axiom, being equivalent to the Axiom of Choice).

\medskip \noindent \emph{Initial ordinal.} For a set $S$, consider all possible well-orderings of $S$. Each of them is isomorphic to a unique ordinal. The corresponding set of ordinals has a minimal element (with respect to the relation $\in$). This element is called \emph{the initial ordinal} that corresponds to $S$.

\medskip \noindent \emph{Cardinals.} Informally, the cardinality of a set $S$ measures its `size'. Formally, the cardinal of $S$, denoted by $|S|$, is identified with the initial ordinal that corresponds to $S$. This allows comparing the `sizes' of any two sets, as any two ordinals are comparable by the relation $\in$. It is clear that we have $|S| \leq |T|$ if and only if there exists an injection $f:S \rightarrow T$.

\medskip \noindent \emph{Arithmetic of infinite cardinals.} We use two basic facts from cardinal arithmetic. If $S,T$ are infinite sets then $|S \times T|=|S \cup T|=\max(|S|,|T|)$. If, in addition, $|S|<|T|$, then $|T \setminus S|=|T|$.

\medskip \noindent \emph{Initial section of an initial ordinal.} It follows from the above definitions that if $O$ is the initial ordinal that corresponds to some set $S$ and $O'$ is an initial section of $O$, then $|O'|<|O|$.

\medskip \noindent \emph{Induction on ordinals and transfinite induction.} The induction principle on ordinals implies the following: Let $O$ be an ordinal, and let $P$ be a property of ordinals. If for any $O' \in O$, we have  $(\forall O''<O': P(O'')) \Rightarrow P(O')$, then $P(O)$. As any well-ordered set corresponds to a unique ordinal, this method can be used to prove claims on general well-ordered sets. Its application is commonly called \emph{transfinite induction}.

\medskip \noindent \emph{Application in our case.} A classical way to use transfinite induction to prove an assertion on a general set $S$ is to use the well-ordering theorem to define a well-ordering $<$ such that $(S,<)$ is isomorphic to the initial ordinal that corresponds to $S$, and then to use transfinite induction to prove the assertion for $(S,<)$. An important feature deployed in this method is that for any initial section $S'$ of $(S,<)$, we have $|S'|<|S|$, as was explained above.

We shall apply this method to the set $S$ all $k$-tuples of points in the lower half-plane (whose cardinality is clearly $\aleph$), and the fact that the cardinality of any initial section in our ordering is strictly smaller than $\aleph$ will play a crucial role in our proof.

\subsection{Proof of the theorem}

In this subsection we prove Theorem~\ref{Thm:Our1}. Let us recall its statement.


\medskip \noindent \textbf{Theorem~\ref{Thm:Our1}.}
For any $2 \leq k \in \mathbb{N}$, there exists a set $T=T(k) \subset \mathbb{R}^2$ such that the following holds:
\begin{enumerate}
\item Any $2k+3$ points of $T$ are seen, through $T$, from a common point $x \in T$.
\item There are $2k+4$ points in $T$ that are not seen, through $T$, from any point of $T$.
\end{enumerate}

We divide the presentation into three parts. First, we reduce the problem to proving the existence of a subset $S$ of the real line that satisfies certain conditions, called a \emph{$k$-shutter}. Then, we describe the transfinite induction argument, without getting into the geometric details. Finally, we present a formal proof that fills in the geometric part.

\subsubsection{Reduction to proving the existence of a $k$-shutter}

Let us denote by $\Re_+^2$ the upper open half-plane $\Re_+^2=\{   (x,y) \in \Re^2 | y>0  \}$, and by $\Re_-^2$ the lower open half-plane $\Re_-^2=\{   (x,y) \in \Re^2 | y<0  \}$. For two points $x \in \Re_+^2,y \in \Re_-^2$, and a subset $S$ of the $x$-axis, $S \subset\{ (x,0) | x \in \Re \}$, we say that \emph{$x$ sees $y$ via $S$}, if $x$ sees $y$ through $\Re_+^2 \cup S  \cup \Re_-^2$.

The following definition plays a central role in the proof of Theorem~\ref{Thm:Our1}.
\begin{definition}
A subset $S$ of the $x$-axis, $S \subset \{(x,0)  | x \in \Re \}$, is called a \emph{$k$-shutter}, if:

\noindent (a) For any $k$-tuple $\{a_1, \ldots , a_k\} \subset \Re_-^2$ there exists a point $z \in \Re_+^2$ that sees each $a_i$ $(1 \leq i \leq k)$ via $S$, and

\noindent (b) There exists a $(k+1)$-tuple $K=\{y_1, \ldots , y_{k+1}\} \subset \Re_-^2$ such that no point in $\Re_+^2$ sees all the points of $K$ via $S$.
\end{definition}

Note that, by symmetry, the definition of a $k$-shutter is not sensitive to interchanging $\Re_+^2$ and $\Re_-^2$. In addition, it is easy to see that for $k \geq 2$, any $k$-shutter $S$ satisfies $|S| \geq 2$ and does not include a segment.

\begin{observation}
\label{obs:shutter}
In order to prove Theorem~\ref{Thm:Our1}, it suffices to prove that for any $k \in \mathbb{N}$, there exists a $k$-shutter $S$.
\end{observation}

\begin{proof}[Proof of Observation~\ref{obs:shutter}]
Given a $k$-shutter $S$, let $T=\Re_+^2\cup \Re_-^2\cup S$, and
let $T'$ be a set of $2k+3$ points of $T$. If $T' \cap S =\emptyset$, then any point of $S$ sees all points of $T'$ through $T$. If $T' \cap S$ consists of a single point $x$, then $x$ sees all the points of $T'$ through $T$. If $|T' \cap S| \geq 2$, then by the pigeonhole principle, either $T' \cap \Re_-^2$ or $T' \cap \Re_+^2$ contains at most $k$ points. W.l.o.g., assume $|T' \cap \Re_-^2| \leq k$. As $S$ is a $k$-shutter, some point $z \in \Re_+^2$ sees $T' \cap \Re_-^2$ through $T$. The point $z$ clearly sees all points of $T'$ through $T$.

On the other hand, take $K_1 \subset \Re_-^2$ to be a $(k+1)$-set that is not seen via $S$ by any point in $\Re_+^2$, and take $K_2 \subset \Re_+^2$ to be a $(k+1)$-set that is not seen via $S$ by any point in $\Re_-^2$. Since $|S| \geq 2$ and $S$ does not include a segment, there exist two points $s_1,s_2 \in S$ that do not see each other through $S$. Then the set $K_1 \cup K_2 \cup \{ s_1,s_2  \} \subset T$ is a $(2k+4)$-sized subset of $T$ that is not seen by any point through $T$.
\end{proof}

\subsubsection{The transfinite induction argument}

To prove the existence of a $k$-shutter, we use transfinite induction. Let $K=\{y_1,\ldots,y_{k+1}\} \subset \Re_-^2$ be a fixed $(k+1)$-set of points in $\Re_-^2$, and let $U = \{ \{a_1,a_2,\ldots,a_k\} \subset \Re_-^2\}$ be the family of all $k$-subsets of $\Re_-^2$. Clearly, $|U|=\aleph$.
Let $O$ be the initial ordinal of cardinality $\aleph$, namely, $|O|=\aleph$ and for each $\lambda \in O$, $|\{     \alpha \in O : \alpha<\lambda      \}|<\aleph$. Let $\lambda \mapsto u_{\lambda}$ ($\lambda \in O$) be a bijection between $O$ and $U$. (Whose existence follows from the well-ordering theorem.) We denote by $<$ the corresponding ordering on $U$.
We shall prove the existence of a $k$-shutter $S$ such that each $k$-set in $U$ is seen via $S$ by some $z \in \Re_+^2$, while no $z \in \Re_+^2$ sees the $(k+1)$-set $K$ via $S$.

We construct $S$ by a transfinite induction process, that corresponds to induction on the ordinal $O$. We index the steps of the process by the elements of $O$, and in step $\lambda$ we `take care' of the $k$-set $a^{\lambda}=\{a_1^{\lambda},\ldots,a_k^{\lambda}\}$ that corresponds (in the isomorphism) to $\lambda \in O$.

During the process, we construct two increasing sequences of subsets of the $x$-axis: $\{A_{\lambda}\}_{\lambda \in O}$ and $\{B_{\lambda}\}_{\lambda \in O}$ (where `increasing' means that if $\lambda_1<\lambda_2$ then $A_{\lambda_1} \subseteq A_{\lambda_2}$ and $B_{\lambda_1} \subseteq B_{\lambda_2}$). The sets $A_{\lambda}$ contain points that will be included in $S$, while the sets $B_{\lambda}$ contain points will not be included in $S$. Of course, we make sure that the $B_{\lambda}$'s are disjoint from the $A_{\lambda}$'s. In addition, we make sure that at each step $\lambda$, we have $|A_{\lambda}|, |B_{\lambda}| \leq \max(|\lambda|,\aleph_0)$.

At Step $\lambda$, we define $A_{\lambda}$ and $B_{\lambda}$ by adding points to the sets $A_{\lambda}^0 = \bigcup_{\lambda'<\lambda} A_{\lambda'}$ and $B_{\lambda}^0=\bigcup_{\lambda'<\lambda} B_{\lambda'}$, respectively. First we add at most $\max(|\lambda|,\aleph_0)$ points to $B_{\lambda}^0$ to form $B_{\lambda}$ in a certain way to be explained shortly. Then we add $k-1$ (or fewer) points to $A_{\lambda}^0$ to form $A_{\lambda}$, in such a way that the set $a^{\lambda}=\{a_1^{\lambda},\ldots,a_k^{\lambda}\}$ is seen through $A_{\lambda}$ by some $z \in \Re_+^2$, while the set $K$ is not seen through $A_{\lambda}$ by any $z \in \Re_+^2$. The points added to $B_{\lambda}$ are responsible for the latter condition -- they are chosen in such a way that forcing $A_{\lambda}$ to avoid them will guarantee that $K$ is not seen via $A_{\lambda}$ from any point in the upper open half-plane.

At the end of the process, we set $S= \bigcup_{\lambda \in O} A_{\lambda}$. It is clear from the construction that any $k$-set $a^{\lambda} \in U$ is seen via $S$ by some $z \in \Re_+^2$ (since it is seen via $A_{\lambda}$) and also that $K$ is not seen via $S$ by any $z \in \Re_+^2$ (as otherwise, it would be seen via some $A_{\lambda}$, contradicting the construction).

The crucial observation that makes the definition of the sets $A_{\lambda}$ and $B_{\lambda}$ possible is that the cardinality of any initial section of $O$ is smaller than $\aleph$, and thus, at any step $\lambda$ of the process we have $|A_{\lambda}^0|, |B_{\lambda}^0|< \aleph$. This means that the sets $A_{\lambda}^0$ and $B_{\lambda}^0$ cover only a `very small' part of the $x$-axis, and so we have `enough room' to select the points we need to add in order to define $B_{\lambda}$ and $A_{\lambda}$, as will be shown formally below.

\subsubsection{Formal proof}

We prove the following proposition.
\begin{proposition}\label{Prop:Formal}
There exist increasing sequences of sets $\{A_{\lambda}\}_{\lambda \in O}$ and $\{B_{\lambda}\}_{\lambda \in O}$, such that, for any $\lambda \in O$, we have:
\begin{enumerate}
\item $A_{\lambda}, B_{\lambda} \subset \{  (x,0) |x \in \Re   \}$;

\item $A_{\lambda} \cap B_{\lambda} = \emptyset$;

\item $|A_{\lambda}|,|B_{\lambda}| \leq \max(|\lambda|,\aleph_0) < \aleph$;

\item The $k$-set $a^{\lambda}=\{a_1^{\lambda},\ldots,a_k^{\lambda}\}$ is seen via $A_{\lambda}$ by some $z \in \Re_+^2$, while the $(k+1)$-set $K=(y_1,\ldots,y_{k+1})$ is not seen via $A_{\lambda}$ by any $z \in \Re_+^2$.
\end{enumerate}
\end{proposition}
Proposition~\ref{Prop:Formal} clearly implies Theorem~\ref{Thm:Our1}, since the set $S=\bigcup_{\lambda \in O} A_{\lambda}$ is a $k$-shutter, as was explained above.

\begin{proof}

The proof is by transfinite induction on the elements of $O$.

\medskip \noindent \emph{Induction basis.} We start by defining $A_0$ and $B_0$. (Note that there is no formal need to present the initial step separately. However, in our case this step is somewhat different from the other ones, and thus we have to present it.)

Consider all lines that are spanned by two points $y_i,y_j \in K $. Let $B_0$ be the set of all intersection points of these lines with the $x$-axis. Clearly, $|B_0| \leq {{k+1}  \choose {2}} $.

In order to define $A_0$, consider all lines that pass through at least one point of $a^0=\{a_1^0, \ldots, a_k^0  \}$ and at least one point of $B_0$.
The union of these lines does not cover the upper half-plane (being a finite union of lines), hence, there exists a point $z \in \Re_+^2$ that is not incident with any such line. Denote the intersections of the segments $[z,a_1^0], \ldots, [z,a_k^0] $ with the $x$-axis by $s_1, \ldots , s_t$ $(1 \leq t \leq k)$, and set $A_0=\{s_1, \ldots , s_t\}$. Note that by the construction of $B_0$,
$A_0 \cap B_0 = \emptyset$.

Both $A_0$ and $B_0$ are finite, and so we have $|A_0|,|B_0| \leq \max(0,\aleph_0)$. The point $z$ sees all $k$ points in $a^0=\{  a_1^0, \ldots a_k^0  \}$ via $A_0$. Furthermore, no point in $\Re_+^2$ can see $K$ via $A_0$. Indeed, $|A_0| \leq k$, and thus, by the pigeonhole principle, if some $z \in \Re_+^2$ sees $K$ via $A_0$, it must see two points $y_i,y_j$ via the same point $s_\ell \in A_0$. In such a case, $s_\ell \in B_0$ by the definition of $B_0$, contradicting the construction which assures that $A_0 \cap B_0 = \emptyset$. Therefore, $A_0$ and $B_0$ satisfy the assertion of the proposition.

\medskip \noindent \emph{Induction step.} For $\lambda \in O$, we describe the definition of $A_\lambda$ and $B_{\lambda}$, assuming that for any $\lambda'<\lambda$, $A_{\lambda'}$ and $B_{\lambda'}$ have been defined and satisfy the assertion. As described above, we denote
\[
A_{\lambda}^0=\cup_{\lambda'<\lambda} A_{\lambda'} \qquad \mbox{ and } \qquad B_{\lambda}^0=\cup_{\lambda'<\lambda} B_{\lambda'}.
\]


\noindent \textbf{Defining $B_{\lambda}$.} Consider the family $\bar{ \L}$ of lines that pass through at least one point of $A_{\lambda}^0$ and at least one of the points of $K=\{y_1, \ldots, y_{k+1}\}$. Let $Z \subset \Re_+^2$ be the set of all intersection points of two lines in $\bar{\L}$ that lie in $\Re_+^2$. 
Fix $z \in Z$. Since $z$ does not see all $k+1$ points of $K$ via $A_{\lambda}^0$ (by the induction hypothesis), there exists $1 \leq i \leq k+1$ such that the intersection point between $[z,y_i]$ and the $x$-axis is not in $A_{\lambda}^0$. We add such an intersection point to $B_{\lambda}^0$ (if there are several such points, we just add one of them arbitrarily). We repeat this process for all $z \in Z$. The resulting set is $B_{\lambda}$. (If $Z= \emptyset$ then $B_{\lambda}=B_{\lambda}^0$.)

Note that since $|A_{\lambda}^0|, |B_{\lambda}^0| \leq \max \{ |\lambda |, \aleph_0   \}$, it follows that $|B_{\lambda}| \leq \max \{|\lambda |, \aleph_0  \}$. Furthermore, by construction, $B_{\lambda} \cap A_{\lambda}^0   =  \emptyset$ and $B_{\lambda}^0 \subset B_{\lambda}$.

\medskip \noindent \emph{Motivation behind the definition of $B_{\lambda}$.} In the definition of $A_{\lambda}$ to be presented below, we add at most $k-1$ points to $A_{\lambda}^0$. As no point $z \in \Re_+^2$ can see more than a single point of $K$ via the same point in $A_{\lambda}$ (due to the definition of $B_0$), a point $z \in \Re_+^2$ can see $K$ via $A_{\lambda}$ only if it sees at least two points of $K$ via $A_{\lambda}^0$. $Z$ is the set of all `dangerous' points in the upper half-plane, that see via $A_{\lambda}^0$ at least two points of $K$. By our definition of $B_{\lambda}$, we ensure that once we define $A_{\lambda}$ such that $A_{\lambda} \cap B_{\lambda} = \emptyset$, no point in $Z$ will see via $A_{\lambda}$ all $k+1$ points of $K$. This guarantees that no point $z \in \Re_+^2$ can see $K$ via $A_{\lambda}$.

\medskip \noindent \textbf{Defining $A_{\lambda}$.} $A_{\lambda}$ is obtained by adding at most $k-1$ points to $A_{\lambda}^0$. These points are chosen such that $a^{\lambda} = \{  a_1^{\lambda}, \ldots ,  a_k^{\lambda}    \} $ will be seen via $A_{\lambda}$ by some point $z$ in the upper open half-plane. Furthermore, the point $z$ will be a point that sees $a_1^{\lambda}$ via $A_{\lambda}^0$. (This condition allows us to add only $k-1$ points, rather than $k$ points.)

Let $x \in A_{\lambda}^0$ be chosen arbitrarily, and let $\ell$ be the line that passes through $x$ and $a_1^{\lambda}$. Let
\[
\L = \{ \ell(a_i^{\lambda},b)  | 2 \leq i \leq k, b \in B_{\lambda}       \}
\]
be the set of all lines $\ell (a_i^{\lambda},b)$ that pass through some point $a_i^{\lambda} \in a^{\lambda}$, $(2 \leq i \leq k)$, and some point $b \in B_{\lambda}$.

Since $|B_{\lambda}|  \leq \max(|\lambda|,\aleph_0) < \aleph$, and since any line in $\L$ intersects $\ell$ in at most one point in the upper open half-plane,
it follows that there exists a point $z \in \ell \cap \Re_+^2$ that is not covered by any line of $\L$.
(Recall that no line in $\L$ coincides with $\ell$, since $\ell$ intersects the $x$-axis in $x \in A_{\lambda}^0$, while any line in $\L$ intersects the $x$-axis in some point of $B_{\lambda}$.)

\medskip The set $A_{\lambda}$ is obtained from $A_{\lambda}^0$ by adding the intersections of the segments $[z, a_i^{\lambda}]$, for $2 \leq i \leq k$, with the $x$-axis. Note that the intersection of $[z, a_1^{\lambda}]$ with the $x$-axis is included in $A_{\lambda}^0$ so there is no need to add it. Also note that it is possible that some of the $k-1$ `added' points already belong to $A_{\lambda}^0$, and so fewer than $k-1$ points are added.

\medskip Obviously, $A_{\lambda}^0 \subset A_{\lambda} $ and $|A_{\lambda}| \leq \max(|\lambda|,\aleph_0)$. By the construction, the added points are not in $B_{\lambda}$, and hence, $A_{\lambda} \cap B_{\lambda} = \emptyset$. Furthermore, the point $z$ sees $a_1^{\lambda}, \ldots ,a_k^{\lambda}$ via $A_{\lambda}$. Finally, the definition of $B_{\lambda}$ guarantees that no $z \in \Re_+^2$ sees $K$ via $A_{\lambda}$. (See the motivation above.) Therefore, $A_{\lambda}$ and $B_{\lambda}$ satisfy the assertion.

\medskip \noindent By transfinite induction, this completes the proof of the proposition, and thus, also of the theorem.
\end{proof}

\begin{remark}
	One can strengthen the assertion of Theorem~\ref{Thm:Our1} by considering a countable dense collection of `forbidden' $(k+1)$-sets, instead of a single $(k+1)$-set $K$. No significant change in the proof method is needed.
\end{remark}

\section{Proof of Theorem~\ref{Thm:Our2}}
\label{sec:Our2}

In this section we prove that no Krasnoselskii-type theorem exists for the notion of visibility through a polygonal path of length $\leq n$ (in short, an \emph{$n$-path}). By constructing a sequence of planar sets $S_{n,k}$, $n,k \geq 2$, each consisting of a finite union of closed segments, we prove:

\medskip

\noindent \textbf{Theorem~\ref{Thm:Our2} (Restatement).} For any $n,k \geq 2$, there exists a set $S = S_{n,k}
\subset \mathbb{R}^2$ such that every $k$ points in $S$ are visible
from a common point through $n$-paths in $S$, but
some $k+1$ points of $S$ are not visible from a common
point through $n$-paths in $S$.

\medskip

\begin{proof}
The proof consists of two steps. In the first step we prove the existence of $S_{2,k}$ for any $k \geq 2$, and in the second step we extend the construction of the first step to obtain $S_{n,k}$ for $n>2$.

\medskip

\noindent \textbf{Step 1: Constructing $S_{2,k}$, $k \geq 2$.}
We obtain $S_{2,k}$ by an appropriate planar embedding of the complete bipartite graph $K_{k+1,k+1}$ with one perfect matching removed. Let $P \subset \Re^2$ be a convex $(2k+2)$-gon. We assume that no three diagonals of $P$ have a common point inside $P$. (This can be achieved by slightly perturbing the vertices of $P$.)

Label the vertices of $P$ cyclically (clockwise) $a_0,b_0,a_1,b_1,\ldots a_{k}, b_{k}$. We regard the indices as numbers modulo $k+1$  (see Figure~\ref{fig:fig1}). Define $\kappa= \lfloor  \frac{k}{2}  \rfloor$.

The matching to be removed will be $$\{[a_i,b_{i+\kappa}]   :  i=0,1,\ldots,k  \}   =  \{[a_{i-\kappa},b_{i}]   :  i=0,1,\ldots,k  \}.  $$
Note that the segments $[a_i,b_{i+\kappa}] $ are diagonals, not boundary edges of $P$.

Define, therefore, for $i=0,1,\ldots,k$, $$B_i=\bigcup\{ [a_{\nu},b_i]  :   \nu=0,1,\ldots,k , \qquad \nu \neq i-\kappa  \},$$ and let $$S_{2,k}  = \bigcup_{i=0}^{k} B_i.$$

$S_{2,2}$ is just a hexagon. $S_{2,3}$ is an octagon with four vertex-disjoint diagonals of order 3. See Figure~\ref{fig:fig1} for $S_{2,4}$ and $S_{2,5}$.

Claims~\ref{cl:star1} and~\ref{cl:star2} below assert that in $S_{2,k}$, every $k$ points are visible from a common point through a 2-path, but there exist $k+1$ points in $S_{2,k}$ that are not visible from any common point through 2-paths.

\begin{figure}
 \centering
    \centering
          \includegraphics[width=.45\columnwidth]{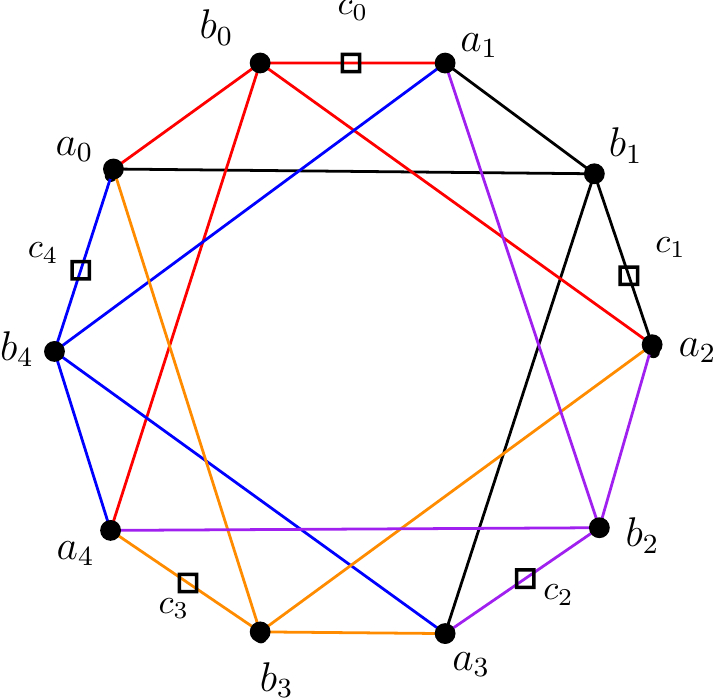}				
 \hspace{0.3cm}
    \centering
        \includegraphics[width=.45\columnwidth]{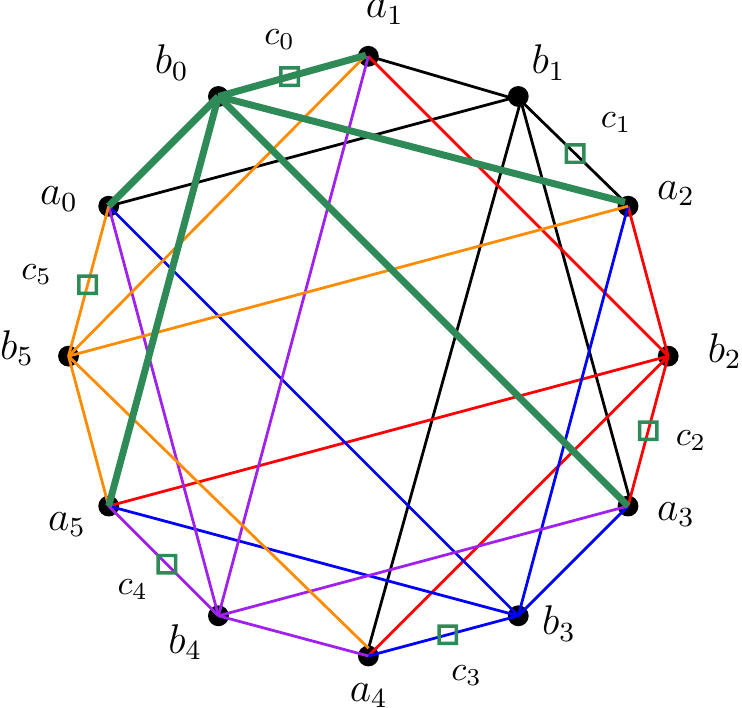}
 \caption{The left figure is $S_{2,4}$ and the right figure is $S_{2,5}$. Each $B_i$ is colored in a different color.}
\label{fig:fig1}
\end{figure}

\begin{claim}
\label{cl:star1}
For any $k$ points $x_1,\ldots,x_k \in S_{2,k}$ there exists a point $z \in S_{2,k}$ that sees all of them through 2-paths in $S_{2,k}$.
\end{claim}
\begin{proof}
Each point $x_i$ $(1 \leq i \leq k)$ belongs to some $B_j$, $0 \leq j \leq k$. (If $x_i$ lies in two $B_j$'s, choose one of them arbitrarily.) Hence we may assume, without loss of generality, that $x_1,\ldots,x_k$ are all contained in $B_1 \cup \ldots \cup B_k$. It follows that $a_{ k+1-\kappa}$ sees each $x_i$ through a 2-path. Indeed, $a_{ k+1-\kappa}$ is
connected by a segment in $S_{2,k}$ to each of $b_1,\ldots,b_k$,
and each $b_i$ is connected by a segment in $S_{2,k}$ to every point in $B_i$.
\end{proof}

\begin{claim}
\label{cl:star2}
For $i=0,1,\ldots,k$,
let $c_i $ be the midpoint of $[b_i,a_{i+1}]$, where the indices are taken modulo $k+1$. Then $c_0,\ldots,c_{k}$ are not visible from any common point in $S_{2,k}$ through 2-paths.
\end{claim}
\begin{proof}
The cases $k=2,3$ are easy to verify. Let $k \geq 4$ and
assume to the contrary that some point $z \in S_{2,k}$ sees all the points $c_0,\ldots,c_{k}$ through 2-paths in $S_{2,k}$.
It follows that for each $0 \leq i \leq k$, $z$ sees $b_i$ or  $a_{i+1}$ through a 1-path, and in particular, the set of vertices of $P$ that $z$ sees through a 1-path is of size $\geq k+1 \geq 5$. We reach a contradiction, by considering three cases:

\medskip \noindent \emph{Case 1: $z$ is a vertex of $P$.}
Thus $z=a_i$ or $z=b_i$ for some $i$, $0 \leq i \leq k$. But $a_i$ sees neither $b_{i+\kappa}$, nor $a_{i+\kappa+1}$ through a 1-path in $S_{2,k}$, and $b_i$ sees neither $b_{i-\kappa-1}$, nor $a_{i-\kappa}$ through a 1-path in $S_{2,k}$.

\medskip \noindent \emph{Case 2: $z$ is not a vertex of $P$, but $z$ lies on a boundary edge of $P$.} In this case, $z$ sees through a 1-path only two vertices of $P$, a contradiction.

\medskip \noindent \emph{Case 3: $z$ is not a vertex of $P$, but $z$ lies on a diagonal of $P$.} In this case, $z$ lies on at most two diagonals of $P$. (Here we use the assumption that no three diagonals of $P$ have a common point inside $P$.) Thus, $z$ sees through a 1-path at most four vertices of $P$. This is a contradiction, as $z$ must see at least 5 vertices of $P$, as was explained above.

\medskip \noindent This completes the proof.
\end{proof}

\noindent \textbf{Step 2: Construction of $S_{n,k}$ for $n>2,k \geq 2$.}
Given $n>2 $ and $k \geq 2$, we modify the construction of $S_{2,k}$ to obtain $S_{n,k}$, in the following way: We consider $S_{2,k}$, and for each $0 \leq i \leq k$ we add a simple polygonal path $\Gamma_i$ of length $n-2$ (with no two consecutive edges collinear) that starts at $c_i$ and lies outside the $(2k+2)$-gon $P$, such that the paths $\Gamma_j$ $(0 \leq j \leq k)$ are pairwise disjoint (see Figure~\ref{fig:fig3}).
For $0 \leq i \leq k$, define $C_i=B_i \cup \Gamma_i$,  and let
$$S_{n,k} = \bigcup_{i=0}^{k}C_i.$$

\begin{figure}
	\centering
	\includegraphics[width=.45\columnwidth]{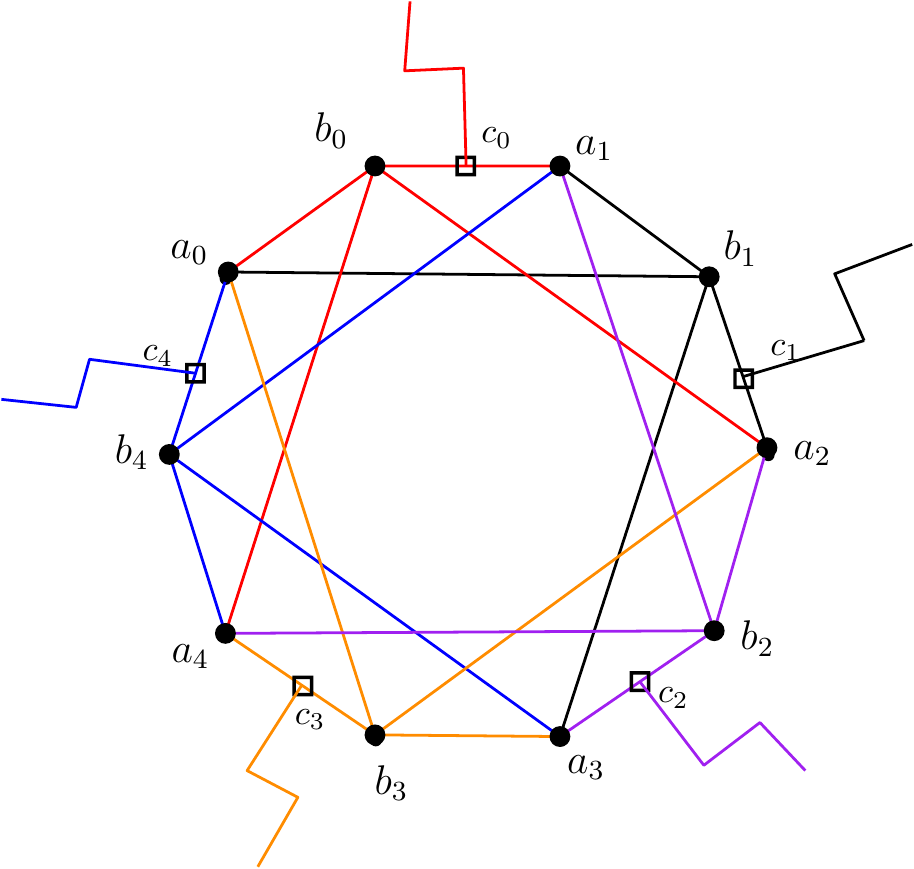}				
	\caption{An illustration of $S_{5,4}$. Each $C_i$ is colored in a different color.} 
\label{fig:fig3}
\end{figure}

\begin{claim}
$S_{n,k}$ satisfies the assertion of Theorem~\ref{Thm:Our2}.
\end{claim}
\begin{proof}
Same as in the proof of Claim~\ref{cl:star1}, given $x_1,\ldots,x_k \in S_{n,k}$ we may assume that $\{x_1,\ldots,x_k \} \subset C_1 \cup \ldots \cup C_k $. In this case, $a_{k+1-\kappa}$ sees each of the points $b_1, \ldots,b_k$ through a 1-path, and thus, sees every point in $C_1 \cup \ldots \cup C_k $ through an $n$-path.

On the other hand, we claim that the $k+1$ outer endpoints of $\Gamma_0 , \ldots , \Gamma_{k} $ are not seen by any common point $x$ through $n$-paths. Indeed, assume on the contrary that some $x \in S_{n,k}$ sees all these points through  $n$-paths. If $x \in S_{2,k}$, then this assumption clearly implies that $x$ sees all the points $c_0, \ldots, c_{k}$ through 2-paths, contradicting Claim~\ref{cl:star2}. If $x \not \in S_{2,k}$ then $x \in \Gamma_i$ for some $i$. In this case, our assumption implies that even $c_i$ sees the $k+1$ outer endpoints of $\Gamma_0 , \ldots , \Gamma_{k} $ through $n$-paths. However, $c_i \in S_{2,k}$, and thus we obtain a contradiction, same as in the first case.
\end{proof}

\noindent This completes the proof of Theorem~\ref{Thm:Our2}.
\end{proof}

\section{Open problems}
The proof method of Theorem~\ref{Thm:Our1} is not constructive, and the resulting set $T$ does not admit any `nice' topological structure (e.g., being a Borel set). Thus, it will be interesting to determine whether one may add some topological restriction on $T$ to the conditions of Theorem~\ref{Thm:Our1}.


\begin{thebibliography}{99}


\bibitem{Bor77} Borwein, J., A proof of the equivalence of Helly's and Krasnosselsky's theorems, {\it Canad. Math. Bull.}, \textbf{20} (1977),
pp.~35--37.

\bibitem{Breen0} M. Breen, Clear visibility, starshaped sets, and
finitely starlike sets, {\it J. of Geometry}, \textbf{19} (1982),
pp.~183--196.

\bibitem{survey} M. Breen, Krasnoselskii-type theorems, {\it Ann. New York Acad. Sci.}, \textbf{440} (1985),
pp.~142--146.

\bibitem{Breen1} M. Breen, Some Krasnosel'skii numbers for
finitely starlike sets in the plane, {\it J. of Geometry},
\textbf{32} (1988), pp.~1--12.

\bibitem{Breen5} M. Breen, Finitely starlike sets whose F-stars have positive measure, {\it J. of Geometry}, \textbf{35} (1989), pp.~19--25.



\bibitem{Breen4} M. Breen, A family of examples showing that no
Krasnosel'skii number exists for orthogonal polygons starhaped via
staircase $n$-paths, {\it J. of Geometry}, \textbf{94} (2009),
pp.~1--6.

\bibitem{Bruckner} A. M. Bruckner and J. B. Bruckner, On $L_n$
sets, the Hausdorff metric, and connectedness, {\it Proc. Amer.
Math. Soc.}, \textbf{13} (1962), pp.~765--767.

\bibitem{Unsolved} H. T. Croft, K. J. Falconer, and R. K. Guy, Unsolved problems in geometry, Volume II, Springer, 1990, pp.~133--136.

\bibitem{EP95} P. Erd\H{o}s and G. Purdy, Extremal problems in discrete geometry, in: R. L. Graham, M. Gr\"{o}tschel, and L. Lov\'{a}sz (ed.), Handbook of Combinatorics, North Holland, 1995.

\bibitem{Horn} A. Horn and F. A. Valentine, Some properties of
$L$-sets in the plane, {\it Duke Math. J.}, \textbf{16} (1949),
pp.~131--140.

\bibitem{Krasnoselskii} M. A. Krasnoselskii, Sur un Crit\`{e}re pour
Qu'un Domain Soit \'{E}toil\'{e}, {\it Math. Sb.}, \textbf{19}
(1946), pp.~309--310.

\bibitem{Levy79} A. Levy, Basic set theory, Springer-Verlag, Berlin, 1979.

\bibitem{Evelyn1} E. Magazanik and M. A. Perles, Generalized convex kernels of simply connected $L_n$ sets in the plane, {\it Israel J. Math.} \textbf{160} (2007), 157--171.

\bibitem{Peterson} B. Peterson, Is there a Krasnonsel'skii theorem
for finitely starlike sets?, Convexity and Related Combinatorial
Geometry, Marcel Dekker, New York, 1982, pp.~81--84.

\bibitem{Valentine} F. A. Valentine, Local convexity and $L_n$
sets, {\it Proc. Amer. Math. Soc.}, \textbf{16} (1965),
pp.~1305--1310.

\end{thebibliography}
\end{document}